\documentclass[reqno]{amsart}
\usepackage{etex}%to solve the confict between xy and pictex
\usepackage{graphicx}
\usepackage{amsmath}
\usepackage{mathrsfs}
\usepackage{amssymb}
\usepackage[all]{xy}
\usepackage{color}
\usepackage{pictex}
\usepackage{subfigure}
\usepackage{xcolor}
\usepackage[colorlinks]{hyperref}
\newtheorem{thm}{Theorem}[section]
\newtheorem{cor}[thm]{Corollary}
\newtheorem{lem}[thm]{Lemma}
\newtheorem{prop}[thm]{Proposition}

\theoremstyle{definition}

\theoremstyle{remark}

\newtheorem{rem}[thm]{Remark}
\numberwithin{equation}{section}
\newcommand{\kf}{K\pi_{4}(F)}

\newcommand{\om}{\Omega}

\newcommand{\autb}{\mathrm{Aut}(B)}

\begin{document}

\title{Finiteness and infiniteness results for Torelli groups of (Hyper-)K\"ahler manifolds }%
\author{Matthias Kreck and Yang Su}%
\address{Mathematisches Institut, Universit\"at Bonn and Mathematisches Institut der Universit\"at Frankfurt}
\email{kreck@math.uni-bonn.de}
\address{HLM, Academy of Mathematics and Systems Science, Chinese Academy of Sciences, Beijing 100190, China}
\address{School of Mathematical Sciences, University of Chinese Academy of Sciences, Beijing 100049, China}
\email{suyang@math.ac.cn}

%\thanks{}%
%\subjclass{}%
%\keywords{}%

\date{}
%\dedicatory{}%
%\commby{}%
% ----------------------------------------------------------------
\begin{abstract} The Torelli group $\mathcal T(X)$  of a closed smooth manifold $X$ is the subgroup of the mapping class group $\pi_0(\mathrm{Diff}^+(X))$ consisting of elements which act trivially on the integral cohomology of $X$. In this note we give  counterexamples to Theorem 3.4 by Verbitsky \cite{V} which states that the Torelli group of simply connected K\"ahler manifolds of complex dimension $\ge 3$ is finite. This is done by constructing under some mild conditions homomorphisms $J: \mathcal T(X) \to H^3(X;\mathbb Q)$ and showing that for certain K\"ahler manifolds this map is non-trivial. We also give a counterexample to Theorem 3.5 (iv) in \cite{V} where Verbitsky claims that the Torelli group of hyperk\"ahler manifolds are finite. These examples are detected by the action of diffeomorphsims on $\pi_4(X)$. Finally we confirm the finiteness result for the special case of the hyperk\"ahler manifold $K^{[2]}$.
\end{abstract}

\maketitle
% ----------------------------------------------------------------
\section{Introduction} Let $X$ be a closed connected oriented manifold. The {\bf mapping class group}  $\mathcal M (X)$ is the group of connected components of the orientation preserving self diffeomorphisms of $X$. The subgroup consisting of diffeomorphisms which act trivially on the integral cohomology ring is called the {\bf Torelli group} of $X$, denote by $\mathcal T (X)$.  In this note we consider two classes of manifolds, the first is characterized by the  following assumptions:

{\bf Assumption (*)}
{\it The manifold $X$ has trivial first rational homology and the first Pontrjagin class (considered in rational cohomology) $p_1(X)\in H^4(X; \mathbb Q)$ is a linear combination of products of $2$-dimensional classes, i.e. there exist classes $z_i \in H^2(X;\mathbb Q)$ and rational numbers $a_{ij}$ such that
 $p_1(X) = \sum a_{ij}z_i \cup z_j$. We denote the set consisting of the classes $z_i$ and of the rational numbers $a_{ij}$ by $D$, like definition data.}

We construct a homomorphism
$$J_D : \mathcal T (X) \to H^3(X;\mathbb Q)
$$
which we consider as a sort of Johnson homomorphism since it is landing in an abelian group and constructed via the mapping torus. The construction is as follows. Consider the mapping torus $X_f$ of $f$. Since $f$ acts trivially on the cohomology groups the Wang sequence looks like:
$$
0 \to H^{k-1}(X; \mathbb Q) \stackrel{\delta}{\rightarrow} H^{k}(X_f; \mathbb Q) \stackrel{i^{*}}{\rightarrow} H^k(X;\mathbb Q) \to 0
$$
Since $H^1(X;\mathbb Q) = 0$  this implies that the inclusion induces an ismomorphism $H^2(X_f;\mathbb Q) \to H^2(X;\mathbb Q)$ and so for each class $z \in H^2(X;\mathbb Q)$ there is a unique class $\bar z \in H^2(X_f;\mathbb Q)$ restricting to $z$. Next we look at $p_1(X_f) - \sum a_{ij} \bar z_i \cup \bar z_j \in H^4(X_f; \mathbb Q)$. By construction it's restriction to $H^4(X;\mathbb Q)$ is zero. Thus there exists a unique class denoted by $J_D(f) \in H^3(X;\mathbb Q)$ mapping to this class in the Wang sequence.

\begin{prop}\label{prop:homo}
Let $X$ fulfil the Assumptions (*), then the map
$$
J_D: \mathcal T(X) \to H^3(X;\mathbb Q)
$$
is a homomorphism.
\end{prop}

Next we are looking for K\"ahler manifolds where this invariant is non-trivial. For this we prove a purely topological result.

\begin{thm}\label{thm:kaehler}
Let $X$ be a simply connected $6$-manifold with $H_2(X;\mathbb Z) \cong \mathbb Z$ and nontrivial cohomology product $H^{2}(X;\mathbb Z) \times H^{2}(X;\mathbb Z) \to H^{4}(X;\mathbb Z)$.  Then the Assumptions (*) are fulfilled and  the invariant is independent of any choices and is denoted by $J$. The image of $J$ is a lattice in $H^3(X; \mathbb Q)$, i.e. a finitely generated abelian group of same rank as the dimension of $H^3(X;\mathbb Q)$.
\end{thm}

Considering hypersurfaces $X(d)$ of degree $d>1$ in $\mathbb CP^4$ we obtain complex $3$-dimensional K\"ahler manifolds with infinite Torelli group, since $d>1$ implies that $H^3(X(d);\mathbb Q) \ne 0$. Taking the product with copies of $\mathbb C \mathrm P^1$ one obtains examples in all complex dimensions $\ge 3$. Namely if $f \in \mathcal T(X(d))$ then by definition of $J_D$ we have for $f \times \mathrm{id}: X(d) \times (\mathbb C \mathrm P^1)^r \to X(d) \times  (\mathbb C \mathrm P^1)^r $ and the obvious choice of $D$
 $$
 J_D(f \times \mathrm{id}) = J(f)
 $$
where we identify $H^3(X(d); \mathbb Q)$ with $H^{3}(X(d)\times (\mathbb C \mathrm P^1)^r; \mathbb Q)$ under the projection map.

Thus we obtain the Corollary contradicting Verbitsky's theorem:

\begin{cor} For all $n\ge 3$ there are K\"ahler manifolds of complex dimension $n$ with infinite Torelli group, actually containing elements of infinite order. In particular taking for example the quintic in $\mathbb CP^4$ we obtain Calabi-Yau manifolds with infinite Torelli group.
\end{cor}

\begin{rem} For simply connected $6$-manifolds $X$ with $H_2(X; \mathbb Z) \cong \mathbb Z$ we compute in \cite{K-S} the full mapping class group as well as the Torelli group and give generators of these groups. In particular we prove that an element $f$ in the Torelli group has finite order if and only if $J(f) = 0$.

\end{rem}

\begin{rem} In general the invariant $J_D$ depends on the choice of the data. So, in general there are many homomorphisms depending on the choice of the data. The simplest example is $X = S^2 \times S^4 \sharp S^3 \times S^3$, where for $z=0$ and $z$ the generator of $H^2(X;\mathbb Z)$ the difference of the $J$-invariants with these data is non-trivial, as shown in \cite{K-S}.
\end{rem}

Now we come to the second class of manifolds which is characterized by the following assumptions:

{\bf Assumption (**)}
{\it The manifolds $X$ are simply connected and $8$-dimensional, whose Betti numbers satisfy one of the following conditions
\begin{enumerate}
\item $b_2 > 2$, $b_3 \ge 1$, $b_4 > b_2(b_2+1)/2$;
\item $b_2=2$, $b_3 > 1$, $b_4 > b_2(b_2+1)/2$;
\item  $b_2=2$, $b_3 =1$, $b_4 > b_2(b_2+1)/2$ and $p_1(X) = 0 \in H^4(X;\mathbb Q)$).
\end{enumerate}
}

\begin{thm}\label{thm:infinite}
If $X$ satisfies the assumptions (**), then the Torelli group  $\mathcal T(X)$ is infinite.
\end{thm}

We apply this result to find a counterexample to Theorem 3.5 (iv) in \cite{V}. For this we consider the complex $4$-dimensional hyperk\"ahler manifold $K^2(T)$. For the construction and the following information we refer to \cite{S} (where it is called $K_2$). The manifold is simply connected and the Betti numbers are: $b_2(K ^2(T)) = 7$, $b_3(K^2(T)) = 8$ and $b_4(K^2(T)) = 108$. Thus our theorem implies:

\begin{cor} The Torelli group of the hyperk\"ahler manifold $K^2(T)$ is not finite.

\end{cor}

This contradicts Verbitsky's Theorem 3.5, iv (\cite {V}).\\

There is another complex $4$-dimensional hyperk\"ahler manifold denoted by $K^{[2]}$, where $K$ is a $K_3$-surface (for the construction see again for example \cite{S}). In this case we show that the Torelli group is finite. This follows from:

\begin{thm}\label{thm:finite}
Let $X$ be a simply-connected  closed $8$-manifold satisfying the following conditions
\begin{enumerate}
\item
$H^{4}(X;\mathbb Q)$ isomorphic to the second symmetric power of $H^2(X;\mathbb Q)$.
\item $H^{3}(X;\mathbb Q)=0$,
\item $p_1(X) \ne 0 \in H^4(X;\mathbb Q)$.
\end{enumerate}
Then $\mathcal T(X)$ is finite.
\end{thm}

\begin{cor} The Torelli group of $K^{[2]}$, where $K$ is a $K_3$-surface, is finite.
\end{cor}

\begin{proof} By \cite{S}  the Betti numbers are $b_2 = 23$, $b_3 =0$ and $b_4 = 276$, which is the same as the dimension of the second symmetric product of $\mathbb Q^{23}$. Since the second symmetric power of the second cohomology of a hyperk\"ahler manifold always injects into the $4$-th cohomology \cite{B}, this implies (1). That $p_1(K^{[2]})$ is non-zero follows from \cite{Be}.
\end{proof}

\begin{rem}This implies that Verbitsky's Theorem 3.5 is true for $K^{[2]}$ and so his main result about the moduli space of hyperk\"ahler manifolds holds for $K^{[2]}$.

\end{rem}

We would like to thank Daniel Huybrechts for bringing Verbitsky's paper \cite{V} to our attention when we told him about our paper \cite{K-S} where we give a complete computation of the mapping class group of certain $6$-manifolds. Since counterexamples to Verbitsky's theorem might be of separate interest we wrote this note. We have sent this note to several people and received a paper by Richard Hain \cite{H} confirming our result that Theorem 3.4 in \cite{V} is incorrect by showing that the Torelli group of certain complex $3$-dimensional K\"ahler manifolds $M$ have an abelian quotient of infinite rank. These examples are different from our examples, they are detected by the induced map  $\pi_3(M)\otimes \mathbb Q  \to \pi_3(M)\otimes \mathbb Q$, whereas for complex $3$-dimensional complete intersections
$\pi_3(X)\otimes \mathbb Q\cong H_3(X; \mathbb Q)$, and so the Torelli groups acts trivial on $\pi_3(X)\otimes \mathbb Q$.

\section{Proofs}

\begin{proof}[Proof of Proposition \ref{prop:homo}] To prove that $J_D$ is a homomorphism we first note that if $X$ is a fibre in $X_f$ then we have an exact sequence with rational coefficients
$$
0 \to H^4(X_f,X) \to H^4(X_f) \to H^4(X) \to 0
$$
and the term on the left is by suspension isomorphism isomorphic to $H^3(X)$ and under this isomorphism $J_D(f)$ corresponds to the class in $H^4(X_f,X)$ which restricts to $ p_1(X_f) - \sum a_{ij} \bar z_i \cup \bar z_j$.

Now we construct $X_{fg}$ from $X_f$ by cutting $I \times X$ along $1/2 \times X$ and regluing it via $g$ from the left to the right. Then we  consider the two fibres over $1/4$ and $3/4$ and denote them by $X_1$ and $X_2$. We consider the restriction map  (with rational coefficients)
$$
H^4(X_{fg}, X_1 \cup X_2) \to H^4(X_{fg}, X_1)
$$
We have an isomorphism
$$H^4(X_f, X) \oplus H^4(X_g, X) \stackrel{\cong}{\rightarrow} H^4(X_{fg}, X_1 \cup X_2)$$
If we identify $H^4(X_{fg},X_1)$ with $H^3(X)$, the homomorphism
$$H^4(X_f, X) \oplus H^4(X_g, X) \stackrel{\cong}{\rightarrow} H^4(X_{fg}, X_1 \cup X_2) \to H^4(X_{fg}, X_1)$$
corresponds to the homomorphism $H^3(X) \oplus H^3(X) \to H^3(X)  $ given by the sum, and the element
$$(p_1(X_f)-\sum a_{ij}\bar z_i \cup \bar z_j, p_1(X_g)-\sum a_{ij}\bar z_i \cup \bar z_j)$$
is mapped  to $p_1(X_{fg})-\sum a_{ij}\bar z_i \cup \bar z_j$. This implies that $J_D(fg) = J_D(f) + J_D(g)$.

\end{proof}

Next we consider the $6$-manifolds in Proposition 1.2. Let $z \in H^{2}(X)$ be a generator, by Poincar\'e duality $H^{4}(X) \cong \mathbb Z$ and thus $z^{2}$ equals to  $d$ times a generator of $H^{4}(X)$. The condition that the cohomology product $H^2(X) \times H^2(X) \to H^4(X)$ is non-trivial is equivalent to saying that $d$ is nonzero. It's shown in \cite{W} and \cite{J} that for every $d \ne 0$ there exist such $X^{6}$. Also there is a connected sum decomposition $X = N \sharp  g(S^{3} \times S^{3})$, where $N$ has the same homology as $\mathbb C \mathrm P^{3}$.

We  introduce a construction of diffeomorphisms, which is a generalization of Dehn twists. Let $S^{3} \times D^{3} \subset X$, $\alpha \in \pi_{3}(SO(4))$, choose a smooth map $\varphi \colon (D^{3},\partial) \to (SO(4),I)$ representing $\alpha$, such that a neighborhood of $\partial D^3$ is mapped to the identity. Define $f \colon X \to X$ by
$$f(u)= \left \{ \begin{array}{cl}
(\varphi (y) \cdot x, y)  & \textrm{if \ } u=(x,y) \in S^{3} \times D^{3} \\
u & \textrm{otherwise} \end{array} \right.$$
We call $f$ a Dehn twist in $S^{3} \times D^{3}$ with parameter $\alpha$.

\begin{proof}[Proof of  Theorem \ref{thm:kaehler}]
The main ingredients of Theorem \ref{thm:kaehler}  already appeared in \cite{K}. The proof of a slightly different situation will be given in  \cite{K-S}. For convenience of the reader we repeat the proof.

By the assumptions the rational Pontrjagin class $p_{1}(X) \in H^{4}(X;\mathbb Q)$ is a multiple of $z^{2}$, say, $p_{1}(X)=l\cdot z^{2}$, $l \in \mathbb Q$. Therefore our invariant $J(f)$ is the pre image of $p_{1}(X_{f})-l \cdot \bar z^{2}$ in $H^{3}(X)$.

In the decomposition $X = N \sharp g(S^{3} \times S^{3})$  we number the standard embedded spheres
$$(S^{3} \times \{pt\})_{1}=S^{3}_{1}, \ (\{pt\} \times S^{3})_{1}=S^{3}_{2}, \ \cdots, (S^{3} \times \{pt\})_{g}=S^{3}_{2g-1}, \ (\{pt\} \times S^{3})_{g}=S^{3}_{2g} $$
They represent a symplectic basis $\{e_{1}, \cdots , e_{2g} \} $ of $H_{3}(X)$.

Let $\alpha \in \pi_{3}(SO(3)) \cong \mathbb Z$ be a generator, $i_{*} (\alpha) \in \pi_{3}(SO(4))$ be its image induced by the inclusion $i \colon SO(3) \to SO(4)$. Then the $4$-dimensional vector bundle $\xi$ over $S^{4}$ corresponding to $i_{*}(\alpha)$ has trivial Euler class and $p_{1}(\xi)$ equals $4$ times a generator of $H^{4}(S^{4})$ (see \cite[Lemma 20.10]{MS}).

Now let $f \colon X \to X$ be the Dehn twist in $S^{3}_{i} \times D^{3}$ with parameter $i_{*} (\alpha)$. Since the Euler class of $i_{*} (\alpha)$ is trivial, it's easy to see that $f$ acts trivially on homology, hence $f \in \mathcal T(X)$. We claim that $J(f)$ equals $4$  times  the Poincar\'e dual of $e_{i}$ in $H^{3}(X)$. This implies that the image of $J$ is a subgroup of $H^{3}(X)$ of the same rank, hence it is a lattice in $H^{3}(X;\mathbb Q)$.

It is enough to show the claim for $i=1$. By the adjunction formula this is equivalent to
$$ \langle J(f), e_i \rangle = 0 \ \  \textrm{for $ i \ne 2$  \ and } \
\langle J(f), e_2 \rangle  = 4.
$$
Since $f|_{S^{3}_{i}}$ is the identity, we have embeddings $S^{3}_{i} \times S^{1} \subset X_{f}$. Let $\bar e_{i}=[S^{3}_{i} \times S^{1}] \in H_{4}(X_{f})$, then $\bar e_{i}$ is a preiamge of $e_{i}$ in the short exact sequence
$$0 \to H_{4}(X) \to H_{4}(X_{f}) \stackrel{\partial}{\rightarrow} H_{3}(X) \to 0$$

There is a decomposition $X_{f} =  (N \times S^{1}) \sharp_{S^{1}} (\sharp_{g}(S^{3} \times S^{3})_{f})$, where $\sharp_{S^{1}}$ denotes the fiber connected sum along $S^{1}$. Then $H^{4}(X_{f}) = H^{4}(N) \oplus H^{4}(\sharp_{g}(S^{3} \times S^{3})_{f})$ and $\bar z^{2}=(z^{2}, 0)$. This implies $\langle \bar z^{2}, \bar e_{i} \rangle =0$.  Therefore we have
$$\langle J(f), e_{i}\rangle = \langle J(f), \partial \bar e_{i} \rangle = \langle \delta J(f), \bar e_{i} \rangle = \langle p_{1}(X_{f}) - l \cdot \bar z^{2}, \bar e_{i} \rangle = \langle p_{1}(X_{f}), [S^{3}_{i} \times S^{1}] \rangle$$
Now notice that $f$ is the identity on a tubular neighborhood of $S^{3}_{i}$'s for $i \ne 2$, therefore the normal bundle of $S^{3}_{i} \times S^{1}$ in $X_{f}$ is trivial, and we have $\langle p_{1}(X_{f}), [S^{3}_{i} \times S^{1}] \rangle =0$ for $i \ne 2$. Let
$$ c \colon S^{3} \times S^{1} \to S^{3} \times S^{1}/S^{3}\times \{pt\} = S^{4}\vee S^{1} \to S^{4}$$
be the quotient map, from the geometric construction it's easy to see that the normal bundle of $S^{3}_{2} \times S^{1}$ in $X_{f}$ equals $c^{*}\xi$. Therefore
$$\langle p_{1}(X_{f}),  [S^{3}_{2} \times S^{1}]\rangle =\langle c^{*}p_{1}(\xi), [S^{3}_{2} \times S^{1}]\rangle = \langle p_{1}(\xi), [S^{4}] \rangle = 4.$$

\end{proof}

The proofs of Theorem \ref{thm:infinite} and Theorem \ref{thm:finite} are based on modified surgery theory \cite{KSD}. We recall some basic definitions here. The normal $k$-type of a smooth oriented manifold $X$ is a fibration $p:B \to BSO$, which is characterized by the assumptions that there is a lift of the normal Gauss map $\nu: X \to BO$ by a $(k+1)$-equivalence $\bar \nu: X \to B$ and that the homotopy groups of the homotopy fibre $F$ are trivial in degree $\ge k+1$ (for details see \cite{KSD}).

To show Theorem \ref{thm:infinite} we will construct infinitely many mapping classes in $\mathcal T(X)$, whose actions on $\pi_{4}(X)$ are pairwisely distinct. Let $p \colon B \to BSO$ be the normal $4$-type of $X$, denote the homotopy fiber of $p$ by $F$, and the kernel of the Hurewicz homomorphism $\pi_4(F) \to H_4(F)$ by $\kf$.

\begin{lem}\label{lem:rank}
Under the assumptions (**) we have
\begin{enumerate}
\item the image of the Hurewicz map $\pi_{4}(B) \to H_{4}(B)$ has rank $>0$;
\item the abelian group $\kf$ has rank $>0$.
\end{enumerate}
\end{lem}

\begin{proof}
We use Sullivan's rational homotopy theory \cite{Su} to show the assumptions (**) imply the conclusions in this lemma. For this we construct the minimal model from information of the cohomology ring.

Let $x_{1}, \cdots, x_{b_{2}}$ be free generators of the minimal model in degree $2$ with differential $0$. Let $y_{1}, \cdots, y_{b_{3}}$ be a part of the free generators in degree $3$, which generate $H^{3}(X;\mathbb Q)$. They have trivial differentials. The elements $x_{i}y_{j}$ generate a space of dimension $b_{2}b_{3}$, on which the differential is trivial. Since by Poincar\'e duality the $5$-th Betti number is $b_{3}$, there must be at least $k=(b_{2}-1)b_{3}$ linearly independent indecomposable elements $z_{1}, \cdots , z_{k}$ in degree $4$, which by the differential are mapped injectively to a subspace of the vector space generated by the $x_{i}y_{j}$. In degree $4$  the dimension of the subspace generated  by $x_{i}x_{j}$ ($i \le j$) is $b_{2}(b_{2}+1)/2$, whereas the  $4$-th Betti number is $b_{4}$, there are additional (to the $z_{i}$) free generators $u_{1}, \cdots, u_{l}$ (where $l \ge b_{4}-b_{2}(b_{1}+1)/2 \ge 1$) with trivial differential in degree $4$. There might be more generators in degree $4$ depending on the product $H^2(X ;\mathbb Q) \otimes H^3(X; \mathbb Q) \to H^5(X ;\mathbb Q)$. But alone from this information the two conclusions in the lemma follow. Namely $\mathrm{Hom}(\pi_4(X), \mathbb Q)$ is isomorphic to the degree $4$ subspace modulo the decomposable elements and so contains the subspace with basis $u_1,...., u_{l}$, $z_1,..., z_{k}$. The dual of the rational Hurewicz homomorphism maps $u_i$ injectively.  Dualizing we see that the dual of the $u_i$ are mapped injectively showing (1) and the dual of the $z_i$ are mapped to $0$ in rational homology, so the rank of $K\pi_{4}(B)=\mathrm{Ker}(\pi_{4}(B) \to H_{4}(B))$ is at least $k=(b_{2}-1)b_{3}$.

Consider the following commutative diagram
$$\xymatrix{
\kf \otimes \mathbb Q \ar[d] \ar[r] & \pi_{4}(F) \otimes \mathbb Q \ar[d] \ar[r] & H_{4}(F ;\mathbb Q) \ar[d] \\
K\pi_{4}(B) \otimes \mathbb Q \ar[r] & \pi_{4}(B)\otimes \mathbb Q \ar[r] & H_{4}(B;\mathbb Q)}$$
where $\pi_{4}(F)\otimes \mathbb Q \to \pi_{4}(B) \otimes \mathbb Q$ is an isomorphism if $p_{1}(X) =0 \in H^{4}(X;\mathbb Q)$ and is injective with a codimension $1$ image if $p_{1}(X) \ne 0 \in H^{4}(X;\mathbb Q)$, and $H_{4}(F;\mathbb Q) \to H_{4}(B;\mathbb Q)$ is injective (by the Leray-Serre spectral sequence of the fibration $F \to B \to BSO$). A little diagram chasing show that $$K\pi_{4}(B) \otimes \mathbb Q \cap \mathrm{Im} (\pi_{4}(F) \otimes \mathbb Q \to \pi_{4}(B) \otimes \mathbb Q) = \kf \otimes \mathbb Q.$$
This show that  under the assumption  $\dim_{\mathbb Q}\kf \otimes \mathbb Q \ge 1$.
\end{proof}

Let $Z^{4}(B; \kf)$ be the group of $4$-cocycles with coefficients in $\kf$,  $\autb$ be the group of fiber homotopy classes of fiber homotopy equivalences of $p \colon B \to BSO$. We will first define a map
$$\Phi \colon Z^{4}(B;\kf) \to \autb$$
as follows.

Let $B^{(4)}$ be the $4$-skeleton of $B$, $e_{1}^{4}, \cdots ,e_{m}^{4}$ be the $4$-cells, then the $4$ dimensional cellular chain group $C_{4}(B)$ is a free abelian group generated by $e_{i}^{4}$ ($i=1, \cdots , m$). Let $q \colon B^{(4)} \to B^{(4)} \vee \vee_{i=1}^{m}S^{4}_{i}$ be the pinch map, where we pinch each $4$-cell $e^{4}_{i}$ to $e^{4}_{i} \vee S^{4}_{i}$. Let $\alpha \colon C_{4}(B) \to \kf$ be a cocycle. For each $\alpha (e_{i}^{4}) \in \kf$, by abuse of notation, we view it as a map $S_{i}^{4} \to B$.  Let $i \colon B^{(4)} \to B$ be the inclusion, define $h \colon B^{(4)} \to B$ to be the composition
$$h=(i \vee \vee_{i=1}^{m} \alpha(e_{i}^{4})) \circ q \colon B^{(4)} \to B^{(4)} \vee \vee_{i=1}^{m}S^{4}_{i} \to B.$$
By the construction $h$ is compatible with the fiber projection $p$, i.~e.~the following diagram commutes up to homotopy
$$\xymatrix{
B^{(4)} \ar[d]^{i}  \ar[r]^{h} & B \ar[d]^{p} \\
B \ar[r]^{p} & BSO}
$$

Now we show that we can extend $h$ to a map on the $5$-skeleton of $B$ compatible with $p$.
Let $e^{5}$ be a $5$-cell, with attaching map $f \colon S^{4} \to B^{(4)}$. By Blakers-Massey theorem we have an isomorphism
$$\pi_{4}(B^{(4)} \vee \vee_{i=1}^{m}S^{4}_{i}) \cong \pi_{4}(B^{(4)}) \oplus \oplus_{i=1}^{m}\pi_{4}(S^{4}_{i}), $$
under this isomorphism the image of $[q \circ f]$ is clearly
$([f], \partial e^{5})$. Since $\alpha$ is a cocyle, $\langle \alpha, \partial e^{5} \rangle =\langle \delta \alpha, e^{5}\rangle =0$, therefore $h$ extends to a map $h \colon B^{(5)} \to B$, compatible with the fiber projection $p$. (The extension may not be unique, we choose one extension.) Since $\pi_{n}(F)=0$ for $n \ge 5$, we can extend $h$ further to a fiber map $h \colon B \to B$.

\begin{lem}\label{lem:h}
$h$ is a fiber homotopy equivalence. $h_{*}$ is the identity on $H_{i}(B)$ for $i \le 4$. $h^{*}$ is the identity on $H^{i}(B)$ for $i \le 3$. $h_{*}$ is the identity on $H_{8}(B;\mathbb Q)$.

\end{lem}
\begin{proof}
It's clear from the construction that $h_{*}$ is the identity on $H_{i}(B)$ for $i \le 3$ and $h^{*}$ is the identity on $H^{i}(B)$ for $i \le 3$. On the chain level $h_{*}$ is
$$h_{*} \colon C_{4}(B) \to C_{4}(B), \ \ e_{i}^{4} \mapsto e_{i}^{4}+ \alpha(e_{i}^{4})$$
where $\alpha(e_{i}^{4})$ is a boundary by definition. Therefore $h_{*}$ is the identity on $H_{4}(B)$. By the Hurewicz theorem we see that  $h_{*} \colon \pi_{i}(B) \to \pi_{i}(B)$ is an isomorphism for $i \le 3$ and is a surjection for $i=4$.  Notice that  $\pi_{4}(B)$ is a finitely generated abelian group, a surjective endomorphism of a finitely generated abelian group must also be injective, therefore $h_{*}$ is an isomorphism on $\pi_{4}(B)$. Since $p \colon B \to B\mathrm{SO}$ is $4$-coconnected, $h$ induces isomorphisms on $\pi_{i}(B)$ for all $i$. Therefore $h$ is a fiber homotopy equivalence.

Since $\pi_{n}(F)=0$ for $n \ge 5$, by rational homotopy theory, there are no indecomposable cohomology classes in $H^{n}(F;\mathbb Q)$ for $n \ge 5$. Therefore  elements in  $H^{8}(F; \mathbb Q)$ are linear combinations of products of cohomology classes of degree $\le 4$. Since $h_{*}$ is the identity on $H_{i}(F)$ for $i \le 4$, $h^{*}$ is the identity on $H^{i}(F;\mathbb Q)$ for $i \le 4$ and hence also the identity on $H^{8}(F;\mathbb Q)$. Thus $h_{*}$ is the identity on $H_{8}(F;\mathbb Q)$. By the Leray-Serre spectral sequence of the fibration $F \to B \to B\mathrm{SO}$, $h_{*}$ is the identity on $H_{8}(B;\mathbb Q)$.
\end{proof}

\begin{proof}[Proof of Theorem \ref{thm:infinite}]
By Lemma \ref{lem:rank} the image of $\pi_{4}(B) \to H_{4}(B)$ has rank $k>0$. Let $x_{1}, \cdots, x_{k}$ be homology classes in the image which generate a free abelian subgroup of rank $k$ in $H_{4}(B)$, $y_{i} \in \pi_{4}(B)$ be a pre-image of $x_{i}$, represented by a map $S^{4} \to B$.  Then we may take $B^{(4)}$ of the form
$$B^{(4)}=X \vee \vee_{i=1}^{k}S^{4}_{i}$$
where $X$ is a $4$-complex, and $[S^4_i]=y_i$ for $i=1,\cdots, k$.

Let $\varphi \in \mathrm{Hom}(H_{4}(B), \kf)$ be a homomorphism, since there are surjective homomorphisms
$$Z^{4}(B, \kf) \to H^{4}(B, \kf) \to \mathrm{Hom}(H_{4}(B), \kf)$$
we may pick a pre-image of $\varphi$, say $\alpha \in Z^{4}(B, \kf) $. Let $h \colon B \to B$ be the image of $\alpha$ under $\Phi$. Consider the action of $h$ on $\pi_{4}(B)$, especially the image of $y_{i}$ under $h_{*}$: let $\iota \colon S^{4} \to B$ be the inclusion of $S^{4}_{i}$, clearly the homotopy class of the composition
$$h \circ \iota  \colon S^{4} \to B^{(4)} \to B^{(4)} \vee \vee_{i=1}^{m}S^{4}_{i} \to B$$
equals $y_{i}+\varphi(x_{i}) \in \pi_{4}(B)$.

By Lemma \ref{lem:rank} there are infinitely many $\varphi \in \mathrm{Hom}(H_{4}(B), \kf)$ such that $\{ \varphi(x_{i}) | \ i=1, \cdots , k \}$ are pairwisely  distinct. Therefore we have constructed infinitely many $h \in \autb$ with properties in Lemma \ref{lem:h}.

Let $\om_{8}(B,p)$ be the $8$ dimensional $B$-bordism group.  It's isomorphic to  the $8$-dimensional stable homotopy group of the corresponding Thom spectrum. By the Atiyah-Hirzebruch spectral sequence there is an isomorphism
$$\om_{8}(B,p) \otimes \mathbb Q \to  H_{8}(B;\mathbb Q)$$
where the image of a bordism class $[N,f] \in \Omega_8(B,p)$ is the image of the fundamental class $f_{*}[N] \in H_8(B;\mathbb Q)$.

Now we fix a normal $4$-smoothing $\bar \nu \colon X \to B$.  Consider the normal $4$-smoothing $\bar \nu \circ h$, where $h \in \autb$ is constructed as above. Since $h_{*}$ is the identity on $H_{8}(B;\mathbb Q)$, we have
$$[X, \bar \nu]=[X, h \circ \bar \nu] \in \om_{8}(B,p) \otimes \mathbb Q$$
Therefore there are infinitely many $h$ (indexed by $i \in \mathbb N$) such that $$[X, h_{i} \circ \bar \nu]=[X, h_{j} \circ \bar \nu] \in \om_{8}(B,p).$$
Let $W_{i}$ be a $B$-bordism between $[X, h_{0} \circ \bar \nu]$ and $[X, h_{i} \circ \bar \nu]$, by the theory of modified surgery, the obstruction to changing $W$ by surgery to an $h$-cobordism is in $L_{9}(\mathbb Z)=0$. Therefore there is a diffeomorphism $f_{i} \colon X \to X$ such that $h_{i} \circ \bar \nu \circ f_i \simeq h_{0} \circ \bar \nu $. Since $\bar \nu$ is a $5$-equivalence, we see that $f_{i*}$ is the identity on $H_{i}(X)$ for $i \le 4$, and $f_{i}^{*}$ is the identity on $H^{i}(X)$ for $i \le 3$. Therefore by Poincar\' e duality, $f_{i*}$ is the identity on $H_{*}(X)$ and hence $f \in \mathcal T(X)$.  But $f_{i*}$ on $\pi_{4}(X)$ are pairwisely distinct. Therefore we have constructed infinitely many elements in $\mathcal T(X)$.
\end{proof}

\begin{proof}[Proof of Theorem \ref{thm:finite}]
Let $p \colon B \to BSO$ be the normal $4$-type of $X$, $F$ be the homotopy fiber of $p$. Fix a normal $4$-smoothing $\bar \nu \colon X \to B$.

We claim that $\pi_3(X)$ and $\pi_4(X)$ are finite. For this we consider the minimal model of $X$. It has generators $x_1,..., x_r$ in degree $2$, where $r$ is the second Betti number. By assumption $x_i \cup x_j$ for $i \le j$ is a basis of $H^4(X;\mathbb Q)$. This and the fact that $H^3(X;\mathbb Q) = 0$  implies that there are no generators in degree $3$ and hence $\pi_3(X)$ is finite. This implies that there are no decomposable elements in degree $5$. Thus the differential on elements of degree $4$ is zero. But then there are no indecomposable elements in degree 4, since they would produce indecomposable cohomology classes. Thus $\pi_4(X)$ is finite. Finally we conclude that also $\pi_3(F)$ and $\pi_4(F)$ are finite. This follows from the homotopy sequence of the fibration $F \to B \to BSO$ since $\pi_i(X) \cong \pi_i(B)$ for $i\le 4$ and $p_1(X) \ne 0$ implying that $\pi_4(B) \to \pi_4(BSO)$ is non-trivial.

Let $f \in \mathcal T(X)$ be a self-diffeomorphism, then $\bar \nu \circ f$ is also a normal $4$-smoothing. Let $\mathcal T_0(X)$ be the subset of $\mathcal T(X)$ consisting of self-diffeomorphisms $f$ such that $\bar \nu \circ f$ and $\bar \nu$ are homotopic as liftings of $\nu$.

\begin{lem} Under our conditions
$\mathcal T_0(X)$ is a finite index subgroup of $\mathcal T(X)$.
\end{lem}
\begin{proof}
It suffices to show that there are finitely many homotopy classes of $\bar \nu \circ f $ over $p$, which in turn follows if  there are finitely many lifts of the Gauss map $\nu$  over $p$. This follows by induction over the skeleta and the Puppe sequence and the fact that $f$ induces the identity on $\pi_2(X) = H_2(X)$ from the fact that the homotopy groups $\pi_3(F)$ and $\pi_4(F)$ are finite and $\pi_i(F) =0$ for $i \ge 5$.
\end{proof}

We finish the proof of the theorem by showing that $\mathcal T_0(X)$ is finite. Given $f \in \mathcal T_0(X)$, choosing a homotopy $h \colon X \times [0,1] \to B$ between $\bar \nu$ and $\bar \nu \circ f$ we obtain a normal $B$-structure $\varphi \colon X_{f} \to B$, where $X_{f} $ is the mapping torus. This represents an element $[X_{f} ,\varphi]$ in $\Omega_{9}(B,p)$. Again by the Atiyah-Hirzebruch spectral sequence we have $\Omega_{9}(B,p) \otimes \mathbb Q \cong H_{9}(B;\mathbb Q)$.
\begin{lem}
$H_9(B;\mathbb Q)=0$.
\end{lem}
\begin{proof}
By the Leray-Serre spectral sequence of the fibration $F \to B \to BSO$, it suffices to show that $H^5(F;\mathbb Q)=0=H^9(F;\mathbb Q)$. But since there are no indecomposable classes in $H^n(F;\mathbb Q)$ for $n \ge 5$, it is enough to show that $H^3(F;\mathbb Q)=0$. Again by the Leray-Serre spectral sequence there is an exact sequence
$$H_4(BSO;\mathbb Q) \stackrel{d}{\rightarrow} H_3(F;\mathbb Q) \to H_3(B;\mathbb Q)=0$$
where $d$ can be identified with the composition
$$H_4(BSO;\mathbb Q) = \pi_4(BSO) \otimes \mathbb Q \stackrel{\partial}{\rightarrow} \pi_3(F) \otimes \mathbb Q \to H_3(F;\mathbb Q).$$
If $p_1(X) \ne 0 \in H^4(X;\mathbb Q)$, then $p_* \colon \pi_4(B) \otimes \mathbb Q \to \pi_4(BSO) \otimes \mathbb Q$ is surjective. This implies $H_3(F;\mathbb Q)=0$.
\end{proof}

Therefore the finitely generated abelian group $\Omega_9(B,p)$ is finite. There are finitely many $f_1, \cdots , f_r \in \mathcal T_0(X)$, together with homotopies $h_1, \cdots, h_r$ such that for any $f \in \mathcal T_0(X)$ and homotopy $h$, $[X_f, \varphi] = [X_{f_i}, \varphi_i]$ for some $1 \le i \le r$. This implies that there is a $B$-bordism $W$ between $X \times [0,1]$ and $X \times [0,1]$, where the two boundary components are identified by $f$ and $f_i$ respectively. The modified surgery obstruction to changing $W$ to a relative $h$-cobordism is $\theta(W) \in L_{10}(\mathbb Z) \cong \mathbb Z/2$. If $\theta (W)=0$, there is a relative $h$-cobordism between $X \times [0,1]$ and $X \times [0,1]$. Such an $h$-cobrodism gives rise to a diffeomorphism $F \colon X \times [0,1] \to X \times [0,1]$ which is a pseudo-isotopy between $f$ and $f_i$. Therefore by Cerf's pseudo-isotopy theorem \cite{Cerf} $f$ is isotopic to $f_i$. If $\theta(W) \ne 0$, then for any $f'$ with this property, we may glue $W$ and $W'$ and get a $B$-bordism $V$ between $X \times [0,1]$ and $X \times [0,1]$ where the two boundary components are identified by $f$ and $f'$ respectively. Now the surgery obstruction $\theta(V)=0$, therefore $f$ and $f'$ are isotopic. This shows that there are at most $2r$ elements in $\mathcal T_0(X)$.
\end{proof}

% ----------------------------------------------------------------
%\bibliographystyle{amsplain}
%\bibliography{99}

\end{document}